\documentclass[a4paper,11pt]{amsart}
\usepackage[utf8]{inputenc}
\usepackage[a4paper,margin=3.00truecm]{geometry}
\usepackage{graphics}
\usepackage{amsmath,amssymb,mathrsfs}
\usepackage{latexsym,bm,dsfont}
\usepackage{enumerate,enumitem}
\usepackage{csquotes}
\usepackage{subcaption}
\usepackage{pgfplots,tikz}
\usetikzlibrary{decorations.pathreplacing}
\pgfplotsset{compat=1.10}
\usepgfplotslibrary{fillbetween}
\usetikzlibrary{patterns}
\usetikzlibrary{intersections, calc, angles}
\usepackage{caption}  

\usepackage{hyperref}
\usepackage{cleveref}
\hypersetup{
	linktocpage=true,
	colorlinks=true,
	linkcolor=blue!70!black,
	citecolor=orange!40!red,
	urlcolor=magenta!80!black,
}
\usepackage[showonlyrefs]{mathtools}
\usepackage{etoolbox,kvoptions,logreq,pdftexcmds}

\numberwithin{equation}{section} 

\newtheorem{theorem}{Theorem}[section]
\newtheorem{corollary}[theorem]{Corollary}
\newtheorem{lemma}[theorem]{Lemma}

\theoremstyle{definition}

\newtheorem{remark}[theorem]{Remark}

\newcommand{\R}{\mathbb{R}}	
\newcommand{\dx}{\,\mathrm{d}x}	
\newcommand{\ds}{\,\mathrm{d}S}	
\renewcommand{\d}{\,\mathrm{d}}
\newcommand{\weak}{\rightharpoonup}

\newcommand{\Om}{\Omega}

\newcommand{\norm}[1]{\left\lVert #1 \right\lVert}

\newcommand\restr[1]{\raisebox{-.5ex}{$|$}_{#1}}
\newcommand\intn{- \hspace{-0.42cm} \int}

\DeclareMathOperator{\dive}{\mathrm{div}}

\usepackage{amsaddr}

\author[G.~Buttazzo]{Giuseppe Buttazzo}
\address[G.~Buttazzo]{%
	Università di Pisa, Dipartimento di Matematica,\\
	Largo Bruno Pontecorvo 5, 56127 Pisa, Italy}
\email[G.~Buttazzo]{giuseppe.buttazzo@unipi.it}

\author[R.~Ognibene]{Roberto Ognibene}
\address[R.~Ognibene]{%
	Università degli Studi di Milano-Bicocca,\\
	Dipartimento di Matematica e Applicazioni \\
	Via Cozzi 55, 20125 Milano, Italy}
\email[R.~Ognibene]{roberto.ognibene@unimib.it}

\title{Asymptotics of nonlinear Robin energies}

\begin{document}
	
	\begin{abstract}
		This paper investigates the asymptotic behavior of a class of nonlinear variational problems with Robin-type boundary conditions on a bounded Lipschitz domain. The energy functional contains a bulk term (the $p$-norm of the gradient), a boundary term (the $q$-norm of the trace) scaled by a parameter $\alpha>0$, and a linear source term. By variational methods, we derive   first-order expansions of the minimum as $\alpha\to 0^+$ (Neumann limit) and as $\alpha\to+\infty$ (Dirichlet limit). In the Dirichlet limit, the energy converges to the one of Dirichlet problem with a power-type quantified rate (depending only on $q$), while the Neumann limit exhibits a dichotomy: under a compatibility condition, the energy linearly approaches the one of Neumann problem, otherwise, it diverges as a power of $\alpha$ depending only on $q$. 
	\end{abstract}
	
\maketitle

\tableofcontents


\medskip


\section{Introduction}\label{sec:intro}

Let $\Om$ be a bounded open and Lipschitz subset of $\R^d$. In this work, we study a class of nonlinear variational problems involving Robin-type boundary conditions. Specifically, we consider the energy functional
\begin{equation}\label{eq:E_a}
E_\alpha=E_\alpha^{f,p,q}:=\inf\left\{\frac1p\int_\Om |\nabla u|^p\dx+\frac{\alpha}{q}\int_{\partial\Om}|u|^q\ds-\int_\Om fu\dx\colon u\in W^{1,p}(\Om)\right\},
\end{equation}
where $\ds$ denotes the $d-1$ dimensional Hausdorff measure on the boundary $\partial\Om$, and $\alpha>0$ is a real parameter. Our main goal is to investigate the asymptotic behavior of $E_\alpha$ as $\alpha\to0^+$ and as $\alpha\to+\infty$, revealing different regimes of interaction between the bulk and boundary terms in the energy.

The parameters $p>1$ and $q>1$ govern the nonlinear character of the problem, both in the interior and on the boundary. The source term $f$ is assumed to belong to the dual space $\big(W^{1,p}(\Om)\big)^*$, so that the integral
$$\int_\Om fu\dx$$
is interpreted as the duality pairing between $f$ and $u$ in the Sobolev space $W^{1,p}(\Om)$. In the following, with a little abuse of notation, we write $\int_\Om f\dx$ instead of $\langle f,1\rangle$.

We emphasize that our analysis is restricted to the case $\alpha>0$, as the behavior for $\alpha<0$ is degenerate. Indeed, in this case the energy $E_\alpha=-\infty$, which can be seen by testing the minimization problem with sequences of constant functions diverging to infinity. This trivializes the minimization problem and motivates our focus on positive values of the parameter $\alpha$.

By standard variational arguments, one can readily show the existence of a unique minimizer for the energy $E_\alpha$. More precisely, for each $\alpha>0$, there exists a unique function $u_\alpha\in W^{1,p}(\Om)\cap L^q(\partial\Om)$ such that $E_\alpha$ is attained at $u_\alpha$. This minimizer satisfies the following nonlinear boundary value problem in the weak sense:
$$\begin{cases}
-\Delta_p u_\alpha=f&\text{in }\Om\\
|\nabla u_\alpha|^{p-2}\partial_\nu u_\alpha+\alpha|u_\alpha|^{q-2}u_\alpha=0&\text{on }\partial\Om,
\end{cases}$$
where $\Delta_p u:=\dive\big(|\nabla u|^{p-2}\nabla u\big)$ denotes the $p$-Laplace operator and $\partial_\nu u$ is the outer normal derivative on $\partial\Om$. The weak formulation of this PDE reads as
$$\int_\Om|\nabla u_\alpha|^{p-2}\nabla u_\alpha\cdot\nabla\varphi\dx+\alpha\int_{\partial\Om}|u_\alpha|^{q-2}u_\alpha \varphi\ds=\int_\Om f\varphi\ds\qquad\text{for all }\varphi\in W^{1,p}(\Om),$$
see also \Cref{lemma:minimizer}.

A natural question concerns the asymptotic behavior of $E_\alpha$ as $\alpha$ varies. As $\alpha\to+\infty$, the boundary term becomes increasingly dominant, enforcing a stronger penalization on the boundary values of $u_\alpha$. One then expects that $u_\alpha$ converges to the solution of the associated Dirichlet problem, and that
$$E_\alpha\to E_\infty:=\inf\bigg\{\frac{1}{p}\int_\Om|\nabla u|^p\dx-\int_\Om fu\dx\ :\ u\in W^{1,p}_0(\Om)\bigg\},$$
as $\alpha\to+\infty$, where $W^{1,p}_0(\Om)$ denotes the Sobolev space of functions vanishing on $\partial\Om$. The unique minimizer $u_\infty\in W^{1,p}_0(\Om)$ solves the Dirichlet problem
$$\begin{cases}
-\Delta_p u_\infty=f&\text{in }\Om\\
u_\infty=0&\text{on }\partial\Om.
\end{cases}$$

On the opposite end, as $\alpha\to0^+$, the influence of the boundary term diminishes. Under the compatibility condition
\begin{equation}\label{eq:compat}
\int_\Om f\dx=0,
\end{equation}
the energy $E_\alpha$ converges to the Neumann energy
$$E_0:=\inf\left\{\frac{1}{p}\int_\Om|\nabla u|^p\dx-\int_\Om fu\dx\ :\ u\in W^{1,p}(\Om)\right\},$$
and minimizers $u_\alpha$ converge to a function $u_0\in W^{1,p}(\Om)$ satisfying the Neumann boundary value problem
$$\begin{cases}
-\Delta_p u_0=f&\text{in }\Om\\
|\nabla u_0|^{p-2}\partial_\nu u_0=0&\text{on }\partial\Om.
\end{cases}$$
To ensure uniqueness of the Neumann solution $u_0$, one typically imposes a normalization condition, such as
$$\int_{\partial\Om}|u_0|^{q-2}u_0\ds=0.$$
see \Cref{sec:neumann} for further discussion.

If the compatibility condition \eqref{eq:compat} is not satisfied, the situation changes drastically: in this case, the minimization problem becomes unbounded from below as $\alpha\to0^+$, that is $E_\alpha\to-\infty$. This can be seen by testing the energy functional \eqref{eq:E_a} with the constant function
$$u\equiv\left(\int_\Om f\dx\right)^{-1}\alpha^{-1/q}$$
which yields the estimate
$$E_\alpha\le\frac{\mathcal{H}^{d-1}(\partial\Om)}{\displaystyle q\left|\int_\Om f\dx\right|^q}-\alpha^{-1/q}\to-\infty\qquad\text{as }\alpha\to 0^+.$$
This divergence reflects the incompatibility of the Neumann condition with a nonzero mean source term, highlighting the delicate interplay between the geometry of $\Om$, the boundary behavior, and the structure of the energy functional.

The main objective of this paper is to derive the   first-order asymptotic expansion of the energy $E_\alpha$ in the two limiting regimes: as $\alpha\to0^+$, which we refer to as the \emph{Neumann limit}, and as $\alpha\to+\infty$, denoted as the \emph{Dirichlet limit}. These asymptotic expansions offer a precise quantitative description of how the Robin-type energy approaches its Neumann or Dirichlet counterparts, depending on the strength of the boundary penalization.

To this end, we begin in \Cref{sec:basic} by establishing some basic properties of the map $\alpha\mapsto E_\alpha$ and of the associated minimizer $u_\alpha$. In particular, we prove existence and uniqueness of a minimizer, its continuity with respect to $\alpha$ and differentiability of the energy $E_\alpha$, together with the explicit expression
\begin{equation*}
	\frac{\d}{\d\alpha}\,E_\alpha=\frac{1}{q}\int_{\partial\Om}|u_\alpha|^q\ds.
\end{equation*}
We then move to the analysis of the Dirichlet and Neumann limits in detail in \Cref{sec:dirichlet} and \Cref{sec:neumann}, respectively.

In the Dirichlet limit regime, as $\alpha\to+\infty$, we show that the boundary term forces the vanishing of the minimizer on $\partial\Om$, and we obtain the following asymptotic expansion
$$E_\alpha=E_\infty-\alpha^{-\frac{1}{q-1}}\,\frac{q-1}{q}\int_{\partial\Om}|\partial_\nu u_\infty|^{\frac{(p-1)q}{q-1}}\ds+o\big(\alpha^{-\frac{1}{q-1}}\big)\qquad\text{as }\alpha\to+\infty,$$
where $u_\infty\in W_0^{1,p}(\Om)$ denotes the unique solution to the Dirichlet problem associated with the energy $E_\infty$.

In contrast, the Neumann limit requires a case-by-case analysis, depending on whether the compatibility condition $\int_\Om f\dx=0$ is satisfied. When this condition holds, we prove that the Robin energy $E_\alpha$ linearly approaches the Neumann counterpart $E_0$, that is
\begin{equation}\label{eq:neumann_th1}
E_\alpha=E_0+\frac{\alpha}{q}\int_{\partial\Om}|u_0|^q\ds+o(\alpha),\qquad\text{as }\alpha\to 0^+,
\end{equation}
where $u_0\in W^{1,p}(\Om)$ is the unique (properly normalized) solution of the Neumann problem associated with $E_0$.

On the other hand, when the compatibility condition is violated, that is $\int_\Om f\dx\ne0$, the behavior of $E_\alpha$ changes drastically. In this case, the energy $E_\alpha$ diverges to $-\infty$ with the following rate
\begin{equation}\label{eq:neumann_th2}
E_\alpha=-\alpha^{-\frac{1}{q-1}}\,\frac{q-1}{q}\frac{\displaystyle\left|\int_\Om f\dx\right|^{\frac{q}{q-1}}}{\mathcal{H}^{d-1}(\partial\Om)^{\frac{1}{q-1}}}+o\left(\alpha^{-\frac{1}{q-1}}\right),\qquad\text{as }\alpha\to0^+.
\end{equation}
We emphasize that this paper is self-contained: all results are proved relying only on elementary variational techniques. The arguments are based on upper and lower energy estimates, obtained either by careful construction of test functions or by exploiting suitable convexity inequalities.

\vspace{0.2cm}

Finally, let us make a brief overview of the state of the art concerning the problem we address. To the best of our knowledge, the   asymptotic behavior of Robin-type energies as the parameter $\alpha$ tends to $0$ or $+\infty$ has been previously investigated only in the linear case $p=q=2$. In particular, among others, we mention \cite{BD} and \cite{BBBT} (and references therein) for a functional-analytic approach to the study of the convergence of resolvents as $\alpha\to+\infty$; we also cite \cite{AR} for the case of exterior three-dimensional domains.
On the other hand, the asymptotic behavior of eigenvalues of the $p$-Laplacian with Robin boundary conditions has received more attention. More precisely, in the linear case $p=q=2$, the limit as $\alpha\to 0$ was studied in \cite[Section 4.1]{BS}, while the limit as $\alpha\to +\infty$ was addressed in \cite[Section 5]{BBBT} and \cite{O24}. We also point out that in \cite{O24}, the author investigates the asymptotic behavior of a variant of the Robin torsional rigidity, in which the forcing term $f$ acts on the boundary $\partial\Om$ rather than in the interior of $\Om$, see \cite[Lemma 3.4]{O24}.
Regarding the first eigenvalue in the nonlinear case $p\neq 2$, to the best of our knowledge, first-order expansions in terms of $\alpha$ have only been established in the range $\alpha\to-\infty$ (see \cite{KP, P}), which is not meaningful in our framework.

\section{Basic properties}\label{sec:basic}

In this section, we gather some fundamental properties of the function $\alpha\mapsto E_\alpha$, which will play a central role in the subsequent analysis.

\begin{lemma}
The map $\alpha\mapsto E_\alpha$ is concave and upper semicontinuous on the interval $(0,+\infty)$.
\end{lemma}

\begin{proof}
This assertion follows directly from the observation that the quantity $E_\alpha$ can be expressed as the infimum over a family of affine functions in the parameter $\alpha$. As is well known, the pointwise infimum of affine functions is concave and upper semicontinuous, which yields the desired conclusion.
\end{proof}

An immediate consequence of this property is the following regularity result.

\begin{corollary}\label{cor:lip}
The function $\alpha\mapsto E_\alpha$ is locally Lipschitz continuous on the interval $(0,+\infty)$.
\end{corollary}

We now turn our attention to the fundamental question of the existence of minimizers for the energy functional associated with our problem. To this end, we introduce the functional
$$J_\alpha(u):=\frac{1}{p}\int_\Om |\nabla u|^p\dx+\frac{\alpha}{q}\int_{\partial\Om}|u|^q\ds-\int_\Om fu\dx$$
defined for functions $u\in W^{1,p}(\Omega)$. It is immediate to verify that the energy level $E_\alpha$ can be characterized variationally as
$$E_\alpha=\inf\left\{J_\alpha(u)\colon u\in \mathcal{W}_{p,q}\right\},$$
where
$$\mathcal{W}_{p,q}:=\left\{u\in W^{1,p}(\Om)\ :\ u\restr{\partial\Om}\in L^q(\partial\Om)\right\}$$
and $u\restr{\partial\Om}$ denotes the trace of $u$, understood in the sense of the trace operator from $W^{1,p}(\Om)$ to $W^{1-\frac{1}{p},p}(\partial\Om)$.

As a first step towards establishing the existence of minimizers, we begin with a classical but essential result concerning the coercivity of the functional $J_\alpha$.

\begin{lemma}[Coercivity]\label{lemma:poinc}
There exists a constant $C=C(d,\Om,p,q)>0$ such that the following inequality holds:
\begin{equation}\label{eq:poinc}
\|u\|_{L^p(\Om)}\le C\left(\|\nabla u\|_{L^p(\Om)}+\|u\|_{L^q(\partial\Om)}\right)
\end{equation}
for all $u\in W^{1,p}(\Om)$. Consequently, the functional $J_\alpha$ is coercive on $W^{1,p}(\Om)$ for every fixed $\alpha>0$, i.e.,
$$\lim_{\norm{u}_{W^{1,p}(\Om)}\to+\infty}J_\alpha(u)=+\infty.$$
\end{lemma}

\begin{proof}
We proceed by contradiction. Suppose that inequality \eqref{eq:poinc} does not hold. Then, there exists a sequence $\{u_n\}_n\subset W^{1,p}(\Om)$ such that
$$\|u\|_{L^p(\Om)}=1\qquad\text{and}\qquad\|\nabla u\|_{L^p(\Om)}+\|u\|_{L^q(\partial\Om)}\to0\quad\text{as }n\to\infty.$$
This implies that, as $n\to\infty$,
$$\begin{cases}
\nabla u_n\to0&\text{strongly in }L^p(\Om),\\
u_n\restr{\partial\Om}\to0&\text{strongly in }L^q(\partial\Om).
\end{cases}$$
Since $\{u_n\}_n$ is bounded in $W^{1,p}(\Om)$, by reflexivity and compact embeddings, there exists a function $\tilde{u}\in W^{1,p}(\Om)$ and a subsequence (not relabeled) such that
$$\begin{cases}
u_n\rightharpoonup\tilde{u}&\text{weakly in }W^{1,p}(\Om),\\
u_n\to\tilde{u}&\text{strongly in }L^p(\Om),\\
u_n\to\tilde{u}&\text{strongly in }L^p(\partial\Om),
\end{cases}$$
as $n\to\infty$. From the strong convergence of the gradients and boundary traces to zero, it follows that $\nabla\tilde{u}=0$ and $\tilde{u}|_{\partial\Om}=0$, hence $\tilde{u}=0$ in $\Om$. On the other hand, the normalization condition $\|u\|_{L^p(\Om)}=1$ and strong convergence in $L^p(\Om)$ yield $\|\tilde{u}\|_{L^p(\Om)}=1$, a contradiction. Thus, \eqref{eq:poinc} must hold.

In order to prove the coercivity of $J_\alpha$, assume by contradiction that for a sequence $\{u_n\}_n$ in $W^{1,p}(\Om)$ we have
$$\norm{u_n}_{W^{1,p}(\Om)}\to+\infty\qquad\text{and}\qquad J_\alpha(u_n)\le C$$
for some constant $C$. Then, combining the estimate
\begin{equation*}
	\frac{1}{p}\int_\Omega |\nabla u_n|^p\dx+\frac{\alpha}{q}\int_{\partial\Omega}|u_n|^q=J_\alpha(u_n)+\int_\Omega fu\dx\leq C+\norm{f}_{(W^{1,p}(\Omega))^*}\norm{u}_{W^{1,p}(\Omega)}
\end{equation*}
with \eqref{eq:poinc} we obtain that both $\norm{u_n}_{L^p(\Om)}$ and $\norm{\nabla u_n}_{L^p(\Om)}$ are bounded, in contradiction with the assumption $\norm{u_n}_{W^{1,p}(\Om)}\to+\infty$.
\end{proof}

We are now ready to state the existence and uniqueness result for minimizers of $E_\alpha$.

\begin{lemma}[Existence and uniqueness of a minimizer]\label{lemma:minimizer}
	For any $\alpha>0$, there exists a unique function $u_\alpha\in W^{1,p}(\Om)$, such that $u_\alpha\restr{\partial\Om}\in L^q(\partial\Om)$,
	achieving $E_\alpha$ and satisfying
	\begin{equation*}
		\begin{cases}
			-\Delta_p u_\alpha=f, &\text{in }\Om, \\
			|\nabla u_\alpha|^{p-2}\partial_{\nu}u_\alpha+\alpha |u_\alpha|^{q-2}u_\alpha =0, &\text{on }\partial\Om,
		\end{cases}
	\end{equation*}
	in a weak sense, that is
	\begin{equation*}
		\int_\Om|\nabla u_\alpha|^{p-2}\nabla u_\alpha\cdot\nabla \varphi\dx+\alpha\int_{\partial\Om}|u_\alpha|^{q-2}u_\alpha\varphi\ds=\int_\Om f\varphi\dx\quad\text{for all }\varphi\in W^{1,p}(\Om).
	\end{equation*}
\end{lemma}
\begin{proof}
	The proof is classical and based on the direct method of the calculus of variations. Indeed, by \Cref{lemma:poinc}, the functional $J_\alpha$ is coercive; moreover, we notice that $J_\alpha$ is weakly lower semicontinuous since the $L^p(\Om)$-norm of the gradient and the $L^q(\partial\Om)$-norm enjoy this property and the last duality term in $J_\alpha$ is linear and bounded by assumption on $f$. Therefore, there exists a minimizer $u_\alpha$. Finally, we observe that the functional is strictly convex (being defined as the sum of two strictly convex terms and a linear one), hence the minimizer is unique.
\end{proof}

We now move our focus to differentiability properties (with respect to $\alpha$) of the minimum $E_\alpha$. We have the following.
\begin{lemma}\label{lemma:diff}
	The map $\alpha\mapsto u_\alpha$ is continuous from $(0,+\infty)$ to $W^{1,p}(\Om)$ and the map
	\begin{equation*}
		\alpha\mapsto \int_{\partial\Om}|u_\alpha|^q\ds
	\end{equation*}
	is continuous on $(0,+\infty)$.	Moreover, the map $\alpha\mapsto E_\alpha$ is of class $C^1$ in $(0,+\infty)$ and there holds
	\begin{equation*}
		\frac{\d}{\d\alpha}\,E_\alpha=\frac{1}{q}\int_{\partial\Om}|u_\alpha|^q\ds.
	\end{equation*}
\end{lemma}
\begin{proof}
	Let $\alpha>0$ be fixed and $\delta_k\to 0$ as $k\to\infty$, so that $\alpha+\delta_k>0$. Since $E_{\alpha+\delta_k}=J_{\alpha+\delta_k}(u_{\alpha+\delta_k})$ is bounded in $k$ (see \Cref{cor:lip}), by the coercivity property \Cref{lemma:poinc} we have that
	\begin{equation*}
		\norm{u_{\alpha+\delta_k}}_{W^{1,p}(\Om)}\leq C,\quad\text{for $k$ sufficiently large,}
	\end{equation*}
	for some $C>0$. In particular, by reflexivity, there exists $v\in W^{1,p}(\Om)$ such that, up to a subsequence
$$\begin{cases}
u_{\alpha+\delta_k}\weak v&\text{weakly in }W^{1,p}(\Om),\\
u_{\alpha+\delta_k}\to v&\text{strongly in }L^p(\Om)~\text{and in }L^p(\partial\Om),
\end{cases}$$
as $k\to\infty$. Moreover, by Fatou lemma, the fact that $f\in (W^{1,p}(\Om))^*$ and \Cref{cor:lip} we have that, up to a subsequence
	\begin{equation*}
		J_\alpha(v)\leq \liminf_{k\to\infty}J_{\alpha+\delta_k}(u_{\alpha+\delta_k})=\lim_{\delta\to 0}E_{\alpha+\delta}=E_\alpha,
	\end{equation*}
	which, by uniqueness of the minimizer, implies that $v=u_\alpha$. Now, being the limit uniquely determined, we have that $u_{\alpha+\delta}\weak u_\alpha$ weakly in $W^{1,p}(\Om)$ as $\delta\to 0$, without passing to subsequences. At this point, we prove that $u_{\alpha+\delta}\to u_\alpha$ strongly in $W^{1,p}(\Om)$ and in $L^q(\partial\Om)$. Let $\delta_k\to 0$ be an arbitrary vanishing sequence. We denote
	\begin{equation*}
		X_k:=\frac{1}{p}\int_\Om|\nabla u_{\alpha+\delta_k}|^p\dx,\qquad Y_k:=\frac{\alpha}{q}\int_{\partial\Om}|u_{\alpha+\delta_k}|^q\ds,\qquad Z_k:=\int_\Om f u_{\alpha+\delta_k}\dx
	\end{equation*}
	and
	\begin{equation*}
		X:=\frac{1}{p}\int_\Om|\nabla u_{\alpha}|^p\dx,\qquad Y:=\frac{\alpha}{q}\int_{\partial\Om}|u_{\alpha}|^q\ds,\qquad Z:=\int_\Om f u_{\alpha}\dx.
	\end{equation*}
	Since $Z_k\to Z$ and since $J_{\alpha+\delta_k}(u_{\alpha+\delta_k})\to J_\alpha(u_\alpha)$ as $k\to\infty$, we deduce that
	\begin{equation*}
		X_k+\left(1+\frac{\delta_k}{\alpha}\right)Y_k\to X+Y\qquad\text{as }k\to\infty.
	\end{equation*}	
	Moreover, one can observe that, for $k$ sufficiently large so that $|\delta_k|\leq \alpha/2$, we have
	\begin{equation*}
		0\leq \frac{Y_k}{2}\leq X_k+ \left(1+\frac{\delta_k}{\alpha}\right)Y_k\leq 2(X+Y),
	\end{equation*}
	and so we have
	\begin{equation}\label{eq:11}
		\lim_{k\to\infty}\left((X_k-X)+(Y_k-Y)\right)=0.
	\end{equation}
	We now claim that 
	\begin{equation}\label{eq:12}
		\lim_{k\to\infty}X_k=X\quad\text{and}\quad\lim_{k\to\infty}Y_k=Y.
	\end{equation}
	Being $\{X_k\}_k$ and $\{Y_k\}_k$ bounded, we have that there exists a subsequence $k_j\to\infty$ as $j\to\infty$, and two real numbers $X^*,Y^*$ such that
	\begin{equation*}
		\lim_{j\to\infty}X_{k_j}=X^*\quad\text{and}\quad\lim_{j\to\infty}Y_{k_j}=Y^*.
	\end{equation*}
	By weak lower semicontinuity of the $p$-norm of the gradient and Fatou lemma for the boundary term, we have that
	\begin{equation*}
		X\leq \liminf_{j\to\infty}X_{k_j}=X^*\quad\text{and}\quad Y\leq \liminf_{j\to\infty}Y_{k_j}=Y^*.
	\end{equation*}
	Let us now assume by contradiction that either $X^*>X$ or $Y^*>Y$. By \eqref{eq:11} we have that
	\[\begin{split}
		0&=\lim_{k\to\infty}\left((X_k-X)+(Y_k-Y)\right)\\
		&=\lim_{j\to\infty}\left((X_{k_j}-X)+(Y_{k_j}-Y)\right)\\
		&= X^*-X+Y^*-Y>0,
	\end{split}\]
	which leads to a contradiction. Hence, $X^*=X$ and $Y^*=Y$ and \eqref{eq:12} holds. This  concludes the proof of the first part.
	
	Let us now compute the first derivative of $E_\alpha$ with respect to $\alpha$. In order to do this, let us test the definition of $E_{\alpha+\delta}$ with $u_\alpha$ and the definition of $E_\alpha$ with $u_{\alpha+\delta}$. We get
	\begin{align*}
		&E_{\alpha+\delta}\leq \frac{1}{p}\int_\Om |\nabla u_\alpha|^p\dx+\frac{\alpha+\delta}{q}\int_{\partial\Om}|u_\alpha|^q\ds-\int_\Om fu_\alpha\dx \\
		&E_{\alpha}\leq \frac{1}{p}\int_\Om |\nabla u_{\alpha+\delta}|^p\dx+\frac{\alpha}{q}\int_{\partial\Om}|u_{\alpha+\delta}|^q\ds-\int_\Om fu_{\alpha+\delta}\dx.
	\end{align*}
	Therefore, we get the bounds
	\begin{equation*}
		\frac{\delta}{q}\int_{\partial\Om}|u_{\alpha+\delta}|^q\ds\leq E_{\alpha+\delta}-E_\alpha\leq \frac{\delta}{q}\int_{\partial\Om}|u_\alpha|^q\ds.
	\end{equation*}
	Now, in view of the continuity result of the first part, we conclude the proof of the second part.
\end{proof}

We conclude this section with the following remark, which will play a useful role in the proof of our main results.

\begin{remark}\label{rmk:g}
Consider the function, defined on $\R$, 
$$g(t):=\frac{a}{q}|t|^q-bt$$
where $a>0$ and $b\in\R$. Then, for every $t\in\R$, the function satisfies the inequality
$$g(t)\ge\min_{\R}g=\frac{1-q}{q}\,\frac{|b|^{q/(q-1)}}{a^{1/(q-1)}}.$$
\end{remark}

\section{The Dirichlet limit}\label{sec:dirichlet}

In the present section, we study the asymptotic behavior of the energy functional $E_\alpha$ as the parameter $\alpha\to+\infty$. In this regime, the functional converges to the classical Dirichlet energy, which we recall to be
$$E_\infty:=\inf\left\{\frac{1}{p}\int_\Om |\nabla u|^p\dx-\int_\Om fu\dx\ :\ u\in W^{1,p}_0(\Om)\right\},$$
achieved by a unique function $u_\infty\in W^{1,p}_0(\Om)$.  We are now ready to establish the   rate of convergence of $E_\alpha$ to $E_\infty$.

\begin{theorem}\label{thm:dir}
	Assume that $|\partial_{\nu} u_\infty|^{\frac{p-q}{q-1}}\partial_{\nu}u_\infty\in W^{1-1/p,p}(\partial\Om)$ and let $\bar{u}\in W^{1,p}(\Om)$ be an extension of $-|\partial_{\nu} u_\infty|^{\frac{p-q}{q-1}}\partial_{\nu}u_\infty$ to $\Om$. Then
	\begin{equation*}
		-\frac{q-1}{q}\int_{\partial\Om}|\partial_{\nu} u_\infty|^{\frac{(p-1)q}{q-1}}\ds\leq\alpha^\gamma(E_\alpha-E_\infty)\leq -\frac{q-1}{q}\int_{\partial\Om}|\partial_{\nu} u_\infty|^{\frac{(p-1)q}{q-1}}\ds+\rho_\alpha,
	\end{equation*}
	where
	\begin{equation*}
		\rho_\alpha:=\frac{\alpha^\gamma}{p}\int_\Om\Big(|\nabla (u_\infty+\alpha^{-\gamma}\bar{u})|^p-|\nabla u_\infty|^p-p\alpha^{-\gamma}|\nabla u_\infty|^{p-2}\nabla u_\infty\cdot\nabla \bar{u}\Big)\dx
	\end{equation*}
	satisfies $\rho_\alpha\to 0$ as $\alpha\to+\infty$ and $\gamma:=1/(q-1)$. Moreover, if $\partial_{\nu}u_\infty \in L^{\frac{(p-1)q}{q-1}}(\partial\Om)$, then
	\begin{equation*}
		E_\alpha=E_\infty-\frac{q-1}{q}\alpha^{-\gamma}\int_{\partial\Om}|\partial_\nu u_\infty|^{\frac{(p-1)q}{q-1}}\ds+o\big(\alpha^{-\gamma}\big),\qquad\text{as }\alpha\to+\infty.
	\end{equation*}
\end{theorem}
\begin{proof}
	First of all, we prove that the following equivalent formulation for $E_\alpha$ holds true:
	\begin{equation}\label{eq:dir_1}
		E_\alpha=E_\infty+\alpha^{-\gamma}R_\alpha,
	\end{equation}
	where the remainder term $R_\alpha$ is given by
	\[\begin{split}
		R_\alpha:=\inf_{u\in W^{1,p}(\Om)}\bigg\{\frac{\alpha^\gamma}{p}\int_\Om\Big(|\nabla(u_\infty&+\alpha^{-\gamma}u)|^p-|\nabla u_\infty|^p\dx-p\alpha^{-\gamma}|\nabla u_\infty|^{p-2}\nabla u_\infty\cdot\nabla u\Big)\dx\\
		&+\int_{\partial\Om}\Big(\frac{1}{q}|u|^{q-2}u+|\nabla u_\infty|^{p-2}\partial_\nu u_\infty\Big)u\bigg\}.
	\end{split}\]
		We first perform the change of variables $u\mapsto u_\infty+\alpha^{-\gamma}u$, with $u\in W^{1,p}(\Om)$, and get that
	\begin{equation*}
		E_\alpha=E_\infty+\widetilde{R}_\alpha,
	\end{equation*}
	with
	\[\begin{split}
		\widetilde{R}_\alpha:=\inf_{u\in W^{1,p}(\Om)}\bigg\{\frac{1}{p}\int_\Om|\nabla(u_\infty+\alpha^{-\gamma}u)|^p\dx&-\frac{1}{p}\int_\Om|\nabla u_\infty|^p\dx\\
		&+\frac{\alpha^{1-\gamma q}}{q}\int_{\partial\Om}|u|^q\ds-\alpha^{-\gamma}\int_\Om fu\dx\bigg\}.
	\end{split}\]
	Then, by using the fact that $1-\gamma q=-\gamma$ and the integration-by-parts identity
	\begin{equation*}
		\int_\Om|\nabla u_\infty|^{p-2}\nabla u_\infty\cdot\nabla u\dx=\int_\Om fu\dx+\int_{\partial\Om}|\nabla u_\infty|^{p-2}\partial_\nu u_\infty u\ds,
	\end{equation*}
	we get that $\widetilde{R}_\alpha=\alpha^{-\gamma}R_\alpha$, which concludes the proof of \eqref{eq:dir_1}. In view of \eqref{eq:dir_1} we are left to prove that
	\begin{equation*}
		-\frac{q-1}{q}\int_{\partial\Om}|\partial_{\nu} u_\infty|^{\frac{(p-1)q}{q-1}}\ds\leq R_\alpha\leq -\frac{q-1}{q}\int_{\partial\Om}|\partial_{\nu} u_\infty|^{\frac{(p-1)q}{q-1}}\ds+\rho_\alpha.
	\end{equation*}
	In order to prove the lower bound, we simply observe that by convexity of the function $t\mapsto |t|^p$ there holds
	\begin{equation*}
		\frac{\alpha^\gamma}{p}\int_\Om|\nabla(u_\infty+\alpha^{-\gamma}u)|^p\dx-\frac{\alpha^\gamma}{p}\int_\Om|\nabla u_\infty|^p\dx \\
		-\int_\Om|\nabla u_\infty|^{p-2}\nabla u_\infty\cdot\nabla u\dx\geq 0
	\end{equation*}
	for all $u\in W^{1,p}(\Om)$ and all $\alpha>0$. Therefore, we have
	\begin{equation*}
		R_\alpha\geq \inf_{u\in W^{1,p}(\Om)}\left\{ \int_{\partial\Om}\left(\frac{1}{q}|u|^{q-2}u+|\partial_{\nu} u_\infty|^{p-2}\partial_{\nu} u_\infty\right)u\ds \right\}.
	\end{equation*}
	At this point, we optimize with respect to $u$ and find (see \Cref{rmk:g}) that the infimum is achieved with $u=\bar{u}$ on $\partial\Om$. Then, by explicit computations we conclude the proof of the lower bound.
	On the other hand, the upper is proved by testing the definition of $R_\alpha$ with $\bar{u}$, while the fact that $\rho_\alpha\to 0$ as $\alpha\to 0^+$ is trivial.
	
	Let us now pass to the proof of the second part. Since $|\partial_{\nu}u_\infty|^{\frac{p-q}{q-1}}\partial_{\nu}u_\infty \in L^q(\partial\Om)$, then there exists $v_n\in W^{1-1/p,p}(\partial\Omega)$ such that
	\begin{equation}\label{eq:dir_2}
		v_n\to -|\partial_{\nu}u_\infty|^{\frac{p-q}{q-1}}\partial_{\nu} u_\infty \quad\text{in }L^q(\partial\Omega)~\text{as }n\to\infty.
	\end{equation}
	Now, let $\bar{v}_n\in W^{1,p}(\Omega)$ be an extension of $v_n$ in $\Omega$. We can replicate the proof of the previous step by using $\bar{v}_n$ in place of $\bar{u}$ and obtain that
\[\begin{split}
-\frac{q-1}{q}\int_{\partial\Om}|\partial_\nu u_\infty|^{\frac{(p-1)q}{q-1}}\ds
&\le\alpha^\gamma(E_\alpha-E_\infty)\\
&\le\tilde{\rho}_{\alpha,n}+\int_{\partial\Om}\left(\frac{1}{q}|v_n|^q+|\partial_\nu u_\infty|^{p-2}\partial_\nu u_\infty v_n\right)\ds,
\end{split}\]
	where
	\begin{equation*}
		\tilde{\rho}_{\alpha,n}:=\frac{\alpha^\gamma}{p}\int_\Om\Big(|\nabla (u_\infty+\alpha^{-\gamma}\bar{v}_n)|^p-|\nabla u_\infty|^p-p\alpha^{-\gamma}|\nabla u_\infty|^{p-2}\nabla u_\infty\cdot\nabla \bar{v}_n\Big)\dx
	\end{equation*}
	satisfies $\tilde{\rho}_{\alpha,n}\to 0$ as $\alpha\to+\infty$. Now we pass to the limit as $\alpha\to+\infty$ and obtain that
	\begin{align*}
-\frac{q-1}{q}\int_{\partial\Om}|\partial_\nu u_\infty|^{\frac{(p-1)q}{q-1}}\ds
&\le\liminf_{\alpha\to+\infty}\alpha^\gamma(E_\alpha-E_\infty)\\
&\le\limsup_{\alpha\to+\infty}\alpha^\gamma(E_\alpha-E_\infty)\\
&\le\int_{\partial\Om}\left(\frac{1}{q}|v_n|^q+|\partial_\nu u_\infty|^{p-2}\partial_\nu u_\infty v_n\right)\ds.
	\end{align*}
Finally, we pass to the limit as $n\to\infty$ and, by the strong convergence \eqref{eq:dir_2}, we obtain
$$\lim_{\alpha\to+\infty}\alpha^\gamma(E_\alpha-E_\infty)=-\frac{q-1}{q}\int_{\partial\Om}|\partial_\nu u_\infty|^{\frac{(p-1)q}{q-1}}\ds,$$
as required.	
\end{proof}


\section{The Neumann limit}\label{sec:neumann}

In this section, we study of the asymptotic behaviour of $E_\alpha$ as $\alpha\to 0^+$. In contrast to the Dirichlet limit, here we distinguish the two cases
\begin{equation*}
	\int_\Om f\dx=0\qquad\text{and}\qquad\int_\Om f\dx\neq 0,
\end{equation*}
which determine whether the Neumann problem
\begin{equation*}
	\begin{cases}
		-\Delta_p u= f, &\text{in }\Om, \\
		|\nabla u|^{p-2}\partial_\nu u=0, &\text{on }\partial\Om,
	\end{cases}
\end{equation*}
has a solution or not.

\subsection{\texorpdfstring{The case $\int_\Om f\dx= 0$}{The zero mean case}}
Let us consider the minimization problem
\begin{equation*}
	E_0:=\inf_{u\in W^{1,p}(\Om)}\Bigg\{\frac{1}{p}\int_\Om|\nabla u|^p\dx-\int_\Om fu\dx\colon u\in W^{1,p}(\Om),~\int_{\partial\Om} |u|^{q-2}u\ds=0\Bigg\},
\end{equation*}
which is the natural limit problem for $E_\alpha$, as $\alpha\to 0^+$. Indeed, by testing the weak formulation of the equation of $u_\alpha$ with a constant test function, we obtain the condition
\begin{equation*}
	\int_{\partial\Om}|u_\alpha|^{q-2}u_\alpha\ds=\frac{1}{\alpha}\int_\Om f\dx=0.
\end{equation*}
Hence, we denote by $u_0\in W^{1,p}(\Omega)$ the unique minimizer of $E_0$, which exists in view of the validity of the compatibility condition
\begin{equation*}
	\int_\Om f\dx=0.
\end{equation*}
Moreover, $u_0$ weakly satisfies
$$\begin{cases}
-\Delta_p u_0=f&\text{in }\Om\\
|\nabla u_0|^{p-2}\partial_\nu u_0=0&\text{on }\partial\Om\\
\int_{\partial\Om}|u_0|^{q-2}u_0\ds=0,
\end{cases}$$
that is
\begin{equation*}
	\int_\Om|\nabla u_0|^{p-2}\nabla u_0\cdot\nabla \varphi\dx=\int_\Om f\varphi\dx\qquad\text{for all }\varphi\in W^{1,p}(\Om).
\end{equation*}
We now prove the main result concerning the asymptotic behavior of $E_\alpha$ in the case $\int_\Om f\dx=0$, that is \eqref{eq:neumann_th1}.
\begin{theorem}
	Assume that $u_0\in L^q(\partial\Om)$. Then
	\begin{equation*}
		E_\alpha=E_0+ \frac{\alpha}{q}\int_{\partial\Om}|u_0|^q\ds+o(\alpha),\quad\text{as }\alpha\to 0^+.
	\end{equation*}
\end{theorem}
\begin{proof}
	By reasoning as in the proof of the second part of \Cref{lemma:diff}, we get that
	\begin{equation}\label{eq:neum_1}
		\frac{1}{q}\int_{\partial\Omega}|u_\alpha|^q\ds\leq \frac{E_\alpha-E_0}{\alpha}\leq \frac{1}{q}\int_{\partial\Omega}|u_0|^q\ds
	\end{equation}
	for all $\alpha>0$. Hence, by weak lower semicontinuity of the $L^q(\partial\Omega)$-norm, the theorem is proved once we prove that
	\begin{equation*}
		u_\alpha\weak  u_0\quad\text{weakly in }L^q(\partial\Omega)~\text{as }\alpha\to 0^+.
	\end{equation*}
	By definition of $E_\alpha$ and in view of \eqref{eq:neum_1} we derive that
	\begin{equation*}
		\frac{1}{p}\int_\Omega|\nabla u_\alpha|^p\dx+\frac{\alpha}{q}\int_{\partial\Omega}|u_\alpha|^q\ds=E_\alpha+\int_\Omega fu_\alpha\dx\leq E_0+\frac{\alpha}{q}\int_{\partial\Omega}|u_0|^q\ds+\int_\Omega fu_\alpha\dx.
	\end{equation*}
	In particular,
	\begin{equation}\label{eq:neum_2}
		\norm{\nabla u_\alpha}_{L^p(\Omega)}^p\leq C+\norm{f}_{(W^{1,p}(\Omega))^*}\norm{u_\alpha}_{W^{1,p}(\Omega)},
	\end{equation}
	for some $C>0$ not depending on $\alpha$ and for $\alpha$ sufficiently small.
	On the other hand, thanks to \Cref{lemma:poinc} and \eqref{eq:neum_1} we get that, for $\alpha$ sufficiently small
	\begin{align*}
		\norm{u_\alpha}_{L^p(\Omega)} &\leq C\left(\norm{\nabla u_\alpha}_{L^p(\Omega)}+\norm{u_\alpha}_{L^q(\partial\Omega)}\right) \\
		&\leq C\left(\norm{\nabla u_\alpha}_{L^p(\Omega)}+\norm{u_0}_{L^q(\partial\Omega)}\right) .
	\end{align*}
	At this point, putting together this last estimate with \eqref{eq:neum_2} and \eqref{eq:neum_1} provides that $\{u_\alpha\}_{\alpha}$ is bounded in $W^{1,p}(\Omega)\cap L^q(\partial\Omega)$ (i.e. in $\mathcal{W}_{p,q}$), as $\alpha\to 0^+$. Therefore, up to a subsequence
	\begin{align*}
		&u_\alpha\weak U \quad\text{weakly in }W^{1,p}(\Omega), \\
		&u_\alpha\weak U \quad\text{weakly in }L^q(\partial \Omega), \\
		&u_\alpha\to U \quad\text{strongly in }L^p(\Omega)\cap L^p(\partial\Omega), 
	\end{align*}
	as $\alpha\to 0$, for some $U\in W^{1,p}(\Omega)$. From the equation satisfied by $u_\alpha$, since 
	\begin{equation*}
		|\nabla u_\alpha|^{p-2}\nabla u_\alpha\weak |\nabla U|^{p-2}\nabla U\quad\text{weakly in }L^{p'}(\Omega,\R^d)
	\end{equation*}
	 as $\alpha\to 0$, we obtain that
	\begin{equation*}
		\int_\Omega |\nabla U|^{p-2}\nabla U\cdot\nabla\varphi\dx=\int_\Omega f\varphi\dx\quad\text{for all }\varphi \in W^{1,p}(\Omega).
	\end{equation*}
	Analogously, since $u_\alpha \weak U$ in $L^q(\partial\Omega)$ as $\alpha\to 0$ we get that
	\begin{equation*}
		|u_\alpha|^{q-2}u_\alpha\weak |U|^{q-2}U\quad\text{weakly in }L^{q'}(\partial\Omega)
	\end{equation*}
	as $\alpha\to 0$, thus implying that
	\begin{equation*}
		0=\lim_{\alpha\to 0}\int_{\partial\Omega}|u_\alpha|^{q-2}u_\alpha\ds=\int_{\partial\Omega}|U|^{q-2}U\ds.
	\end{equation*}
	In particular, by uniqueness of the minimizer, we have that $U=u_0$. Then, since the limit $U$ is uniquely determined, we have that 
	\begin{equation*}
		u_\alpha\weak u_0\quad\text{weakly in }L^q(\partial\Omega)~\text{as }\alpha\to 0,
	\end{equation*}
	without passing to subsequences. By \eqref{eq:neum_1} and weak lower semicontinuity of the $L^q(\partial\Omega)$-norm, we conclude the proof.
\end{proof}

%

\subsection{\texorpdfstring{The case $\int_\Om f\dx\neq 0$}{The non-zero mean case}}

In this case, the expected limit Neumann problem does not have a solution and, as observed in the introduction, the Robin energy $E_\alpha$ diverges to $-\infty$ as $\alpha\to 0^+$. In the next result, we detect its   asymptotic behavior, thus proving \eqref{eq:neumann_th2}.

\begin{theorem}
	There holds
	\begin{equation*}
		-\frac{q-1}{q}\frac{\displaystyle \left|\int_\Om f\dx\right|^{\frac{q}{q-1}}}{\mathcal{H}^{d-1}(\partial\Om)^{\frac{1}{q-1}}}+\alpha^{\gamma}K_f\leq \alpha^\gamma E_\alpha\leq -\frac{q-1}{q}\frac{\displaystyle \left|\int_\Om f\dx\right|^{\frac{q}{q-1}}}{\mathcal{H}^{d-1}(\partial\Om)^{\frac{1}{q-1}}},
	\end{equation*}
	where $\gamma:=1/(q-1)$ and
	\begin{equation*}
		K_f:=\inf\left\{\frac{1}{p}\int_\Om|\nabla v|^p\dx-\int_\Om fv\dx\ :\ v\in W^{1,p}(\Om),~\int_{\partial\Om}v\ds=0\right\}.
	\end{equation*}
	In particular,
	\begin{equation*}
		E_\alpha=-\frac{1}{\alpha^\gamma}\frac{q-1}{q}\frac{\displaystyle \left|\int_\Om f\dx\right|^{\frac{q}{q-1}}}{\mathcal{H}^{d-1}(\partial\Om)^{\frac{1}{q-1}}}+o\left(\frac{1}{\alpha^\gamma}\right),\quad\text{as }\alpha\to 0^+.
	\end{equation*}
\end{theorem}
\begin{proof}
	The upper bound can be easily obtained by testing $E_\alpha$ with a constant $u\equiv t\in\R$. Indeed, we get $E_\alpha\leq g(t)$, with
	\begin{equation*}
		g(t):=\frac{\alpha\mathcal{H}^{d-1}(\partial\Om)}{q}|t|^q-t\int_\Om f\dx,
	\end{equation*}
	and so, by \Cref{rmk:g} we deduce the upper bound.
	
	For what concerns the upper bound, we let
	\begin{equation*}
		v_u:=u-\intn_{\partial\Om} u\ds,
	\end{equation*}
	so that $u=t_u+v_u$ with
	\begin{equation*}
		t_u:=\intn_{\partial\Om} u\ds. 
	\end{equation*}
	With this change, thanks to the convexity inequality
	\begin{equation*}
		|t+v|^q\geq |t|^q+q|t|^{q-2}tv
	\end{equation*}
	and the fact that
	\begin{equation*}
		\int_{\partial\Om}v_u\ds=0,
	\end{equation*}
	we have that
	\begin{align}
		E_\alpha&=\inf\left\{\frac{1}{p}\int_\Om|\nabla v_u|^p\dx+\frac{\alpha}{q}\int_{\partial\Om}|t_u+v_u|^q\ds-\int_\Om f v_u\dx-t_u\int_\Om f\dx\colon u\in W^{1,p}(\Om)\right\}\notag \\
		&\geq \inf\Bigg\{\frac{1}{p}\int_\Om|\nabla v_u|^p\dx+\frac{\alpha}{q}\int_{\partial\Om}|t_u|^q\ds+\alpha |t_u|^{q-2}t_u\int_{\partial\Om}v_u\ds \notag\\
		&\qquad\quad-\int_\Om f v_u\dx-t_u\int_\Om f\dx\colon u\in W^{1,p}(\Om)\Bigg\}\notag \\
		&\geq \inf\left\{\frac{1}{p}\int_\Om|\nabla v|^p\dx-\int_\Om f v\colon v\in W^{1,p}(\Om),~\int_{\partial\Om}v\ds=0\right\} \notag\\
		&+\inf\left\{\frac{\alpha |t|^q}{q}\mathcal{H}^{d-1}(\partial\Om)-t\int_\Om f\dx\colon t\in\R\right\}.\label{eq:neumann_f_not_0_1}
	\end{align}	
	Now, by classical variational methods, one can easily see that the first term
	\begin{equation*}
		K_f:=\inf\left\{\frac{1}{p}\int_\Om|\nabla v|^p\dx-\int_\Om f v\colon v\in W^{1,p}(\Om),~\int_{\partial\Om}v\ds=0\right\}>-\infty
	\end{equation*}
	is achieved by a unique function $v\in W^{1,p}(\Om)$ satisfying
	\begin{equation*}
		\begin{cases}
			-\Delta_p v=f, &\text{in }\Om, \\
			\displaystyle|\nabla v|^{p-2}\partial_\nu v =-\frac{1}{\mathcal{H}^{d-1}(\partial\Om)}\int_\Om f\dx, &\text{on }\partial\Om, \\
			\displaystyle \int_{\partial\Om} v\ds=0. &
		\end{cases}
	\end{equation*}
	Finally, in view of \Cref{rmk:g}, we can explicitly write the second term \eqref{eq:neumann_f_not_0_1} and conclude the proof.
\end{proof}

\section*{Acknowledgments}
The work of G. Buttazzo is part of the project 2017TEXA3H {\it``Gradient flows, Optimal Transport and Metric Measure Structures''} funded by the Italian Ministry of Research and University. 

R. Ognibene was partially supported by the European Research Council (ERC), through the European Union’s Horizon 2020 project ERC VAREG - Variational approach to the regularity of the free boundaries (grant agreement No. 853404). Part of the work was carried out during a visit of R. Ognibene at the University of Pisa, which is gratefully acknowledged.

G. Buttazzo and R. Ognibene are members of the Gruppo Nazionale per l'Analisi Matematica, la Probabilit\`a e le loro Applicazioni (GNAMPA) of the Istituto Nazionale di Alta Matematica (INdAM).

\bibliographystyle{aomalpha}

\bibliography{biblio}

\end{document}